\documentclass[letterpaper, 10 pt,conference]{ieeeconf}  

\IEEEoverridecommandlockouts                              

\usepackage[utf8]{inputenc}
\usepackage{amssymb}
\usepackage{amsmath}
\usepackage{bm}
\usepackage{hyperref}
\usepackage{tikz}
\usepackage{tkz-euclide}
\usepackage{chngcntr}
\usepackage{verbatim}
\usepackage{mathtools}
\usepackage{esvect}
\usepackage{subfiles}
\usepackage{color}
\usepackage{xspace}
\usepackage{xparse}
\usepackage[noend]{algpseudocode}
\usepackage{subfigure}
\usepackage{bbm}
\usepackage{enumerate}
\usepackage{graphicx}
\usepackage{multirow}

\usepackage{pgfplotstable}
\usepackage{pgfplots}
\pgfplotsset{
	table/search path={plot_figures},
}
\usepackage{tikz}
\usepackage{xr}
\usetikzlibrary{external}
\usetikzlibrary{arrows,%
	petri,%
	topaths}%
\usetikzlibrary{graphs,graphs.standard}
\usepackage{tkz-berge}

\usepackage[percent]{overpic}

\usepackage{epsfig} 
\newtheorem{assumption}{Assumption}

\newtheorem{theorem}{Theorem}
\newtheorem{lemma}[theorem]{Lemma}
\newtheorem{proposition}[theorem]{Proposition}
\newtheorem{corollary}[theorem]{Corollary}
\newtheorem{remark}{Remark}

\usepackage{pgfplotstable}
\usepackage{pgfplots}
\usepackage{cite}
\pgfplotsset{
	table/search path={plot_figures},
}
\usepackage{tikz}
\usepackage{xr}
\usetikzlibrary{arrows,%
	petri,%
	topaths}%
\usetikzlibrary{graphs,graphs.standard}
\usepackage{tkz-berge}
\usepackage{subfigure}

\allowdisplaybreaks

\title{\LARGE \bf
	Non-Bayesian Social Learning with Uncertain Models over Time-Varying Directed Graphs
}

\author{C\'esar A. Uribe, James Z. Hare, Lance Kaplan, and Ali Jadbabaie 
	\thanks{This research was sponsored by the DARPA Lagrange, Vannevar Bush Fellowship, and OSD LUCI programs. The views and conclusions contained in this document are those of the authors and should not be interpreted as representing the official policies, either expressed or implied, of the Army Research Laboratory or the U.S. Government. The U.S. Government is authorized to reproduce and distribute reprints for Government purposes notwithstanding any copyright notation herein.}
	\thanks{C.A.U and A.J. are with the Laboratory for Information and Decision Systems (LIDS), and the Institute for Data, Systems, and Society (IDSS),
		Massachusetts Institute of Technology, Cambridge, MA
		(\textit{\{cauribe,jadbabai\}@mit.edu}). J.Z.H. (\textit{james.z.hare31@gmail.com}) and L.K. (\textit{lance.m.kaplan@us.army.mil}) are with the Army Research Laboratory, Adelphi, MD . }%
}

\begin{document}

	\maketitle
	\thispagestyle{empty}
	\pagestyle{empty}
	
	\begin{abstract}
		We study the problem of non{-}Bayesian social learning with uncertain models, in which a network of agents seek to cooperatively identify the state of the world based on a sequence of observed signals. In contrast with the existing literature, we focus our attention on the scenario where the statistical models held by the agents about possible states of the world are built from finite observations. We show that existing non{-}Bayesian social learning approaches may select a wrong hypothesis with non{-}zero probability under these conditions. Therefore, we propose a new algorithm to iteratively construct a set of beliefs that indicate whether a certain hypothesis is supported by the empirical evidence. This new algorithm can be implemented over time{-}varying directed graphs, with non{-}doubly stochastic weights. 
	\end{abstract}
	
	\section{Introduction}
	
	Non-Bayesian social learning has emerged as a topic of interest as it captures a variety of computational and cognitive constraints and biases that agents may have in a making cooperative decisions~\cite{jad12,molavi2018theory,eps10}. In contrast to Bayesian approaches~\cite{ace11,gal03}, non{-}Bayesian learning assumes agents have bounded rationality and the information aggregation mechanism differs from the Bayesian setting. The main objective is then to design belief update rules that iterative aggregate information (e.g. weighted averages) from neighbors and private observations, and asymptotically behaves as if all information was available at centrally~\cite{ned15,sha14,ned17e}.
	
	Formally, given a state of the world $\theta^*$, a group of agents following a non{-}Bayesian social learning approach, receive private observations and communicate with other agents in the network to agree on a state, from a set $\Theta$, that best explains the observed signals. This task is achieved by using statistical models for each member of the hypothesis set and sequentially testing whether the observed signals are distributed according to one of the models, which in turn corresponds to one possible state of the world. In~\cite{ned15,ned17e}, this process is described as a group of $m$ agents trying to solve collectively the following optimization problem:
	\begin{align}\label{prob:main}
	\min_{\theta \in \Theta} \sum_{i=1}^{m}D_{KL}\left(P_{\theta^*}^i\|P^i_\theta\right),
	\end{align}
	where each agent sequentially receives observations from a random variable distributed according to an \textit{unknown} $P_{\theta^*}^i$, $P^i_\theta$ is the statistical model an agent $i$ holds about hypothesis $\theta$, and $D_{KL}(P\|Q)$ is the Kullback{-}Leibler divergence between distributions $P$ and $Q$.
	
	The solution of problem~\eqref{prob:main} has been extensively studied for different graph connectivity assumptions, topologies, observation models, robustness, for which consistency, and both asymptotic and non{-}asymptotic convergence rates has been established~\cite{ned16c,uribe2018increasing,sha13,lal14b,rah15,su16c,ned15b,ned16}. However, one required property for existing results is that every agent knows the statistical models corresponding to each of the hypothesis in the hypotheses set precisely. This is referred to as \textit{dogmatic} knowledge of the hypotheses. That is, for each $\theta \in \Theta$, every agent can evaluate the likelihood of an observed signal given that $\theta$ is the state of the world. As a result, existing algorithms are shown to concentrate beliefs asymptotically around the hypotheses that solves~\eqref{prob:main}.
	
	In this paper, we focus on the case where the statistical models of each of the hypotheses are built from finite empirical evidence. Thus, there is uncertainty about the models. For example, the probability of certain events under a hypothesis must be estimated from finite amounts of data. Agents need to consider statistical uncertainty about the likelihood functions. Taking such uncertainty into account has been previously studied in possibility theory \cite{DP2012}, probability intervals \cite{W1997}, fuzzy set theory \cite{Z1965}, and belief functions \cite{S1976, SK1994} where the likelihood function parameters are expressed within a fixed interval. In this work we model uncertainty in the likelihood function parameters as a second-order probability density function~\cite{W1996, J2018}. Such that the parameters are uniformly distributed when the prior evidence is 0 and becomes concentrated around the ground truth distribution as more prior evidence is collected. 
		
	Figure~\ref{fig:toy} presents a toy example, on a centralized scenario that will help us motivate and explain the distributed learning problem with uncertain models. Consider you are given a finite set of labeled and weighted dice, where each die corresponds to a possible state of the world. You are allowed to roll each die a finite number of times, not necessarily the same number for each die. One can build a statistical model for each die from the histograms generated by the observed outcomes. Then, you are given a new unlabeled die, corresponding to the current state of the world, and you are asked to assign this new die to one of the classes of the observed dice in the previous stage. Clearly, if one has an access to the outcomes of an infinite number of rolls of the original dice set, a perfect statistical model can be built, and the task reduces to classical hypothesis testing. Given that one only has access to a possibly non{-}uniform finite number of realizations, the built statistical models are themselves uncertain, and such uncertainty must be taken into account in order to select the state of the world. 
	
	\vspace{0.2cm}
	\begin{figure}[t]
		\centering
		\resizebox{!}{0.2\textwidth}{
		\begin{overpic}[width=0.4\textwidth]{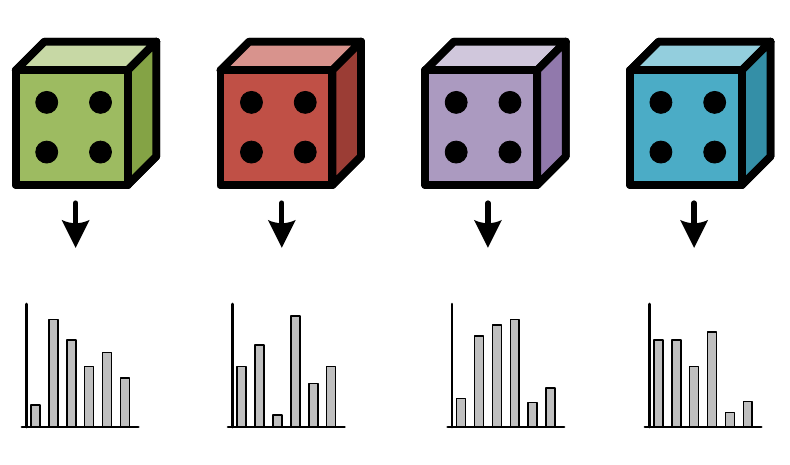}
			\put(46,61){{\small $\text{\underline{\textbf{STAGE 1}}}$}}
			\put(2,23){{\small $R_{\theta_1} \text{rolls}$}}
			\put(28,23){{\small $R_{\theta_2} \text{rolls}$}}
			\put(55,23){{\small $R_{\theta_3} \text{rolls}$}}
			\put(80,23){{\small $R_{\theta_4} \text{rolls}$}}
			\put(9,56){{\small $\theta_1$}}
			\put(34,56){{\small $\theta_2$}}
			\put(61,56){{\small $\theta_3$}}
			\put(86,56){{\small $\theta_4$}}
			\put(-4,-1){\rotatebox{90}{ $\text{Histograms}$}}
		\end{overpic}} 
		\resizebox{!}{0.2\textwidth}{
		\begin{overpic}[width=0.4\textwidth]{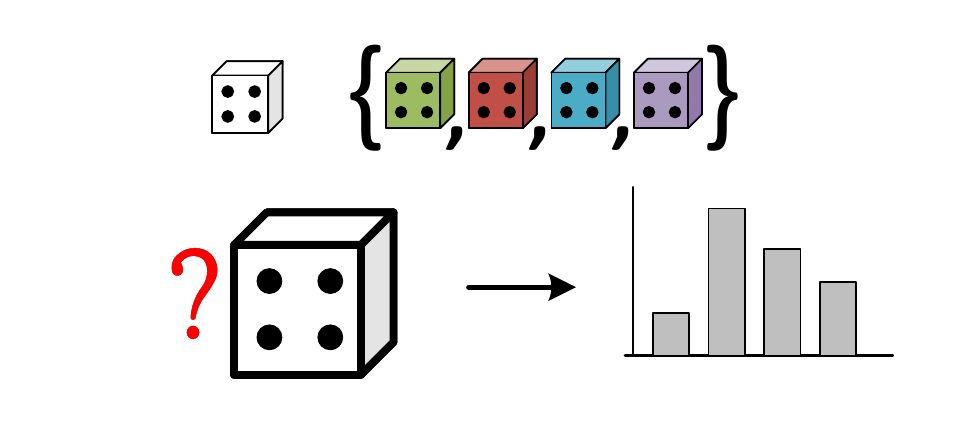}
			\put(46,42){{\small $\text{\underline{\textbf{STAGE 2}}}$}}
			\put(30,33){{{\Large $\boldsymbol{\in}$}}}
			\put(43,23){{\small $\text{Sequential }$}}
			\put(48,18){{\small $\text{Rolls}$}}
			\put(67,3){{\small ${\theta_1} $}}
			\put(73,3){{\small ${\theta_2} $}}
			\put(79,3){{\small ${\theta_3} $}}
			\put(85,3){{\small ${\theta_4} $}}
			\put(85,20){{\small $\mu_k(\theta)$}}
		\end{overpic} 
		}
		\caption[A toy example of non{-}Bayesian learning with uncertain models]{\textbf{A toy example of non{-}Bayesian learning with uncertain models.} One is given four dice (green, red, purple, and blue), each representing a class or possible state of the world from a set $\Theta = \{\theta_1, \theta_2,\theta_3,\theta_4\}$. Each of the dice can be rolled a finite number of times ($R_{\theta_1},R_{\theta_2},R_{\theta_3},R_{\theta_4}$) and the outcomes are recorded in histograms. Then, a new die is given, corresponding to the state of the world $\theta^*$, and one is tasked to identify it by assigning this new die to one state from $\Theta$ from sequential rolls of the new die.}
		\label{fig:toy}
	\end{figure}
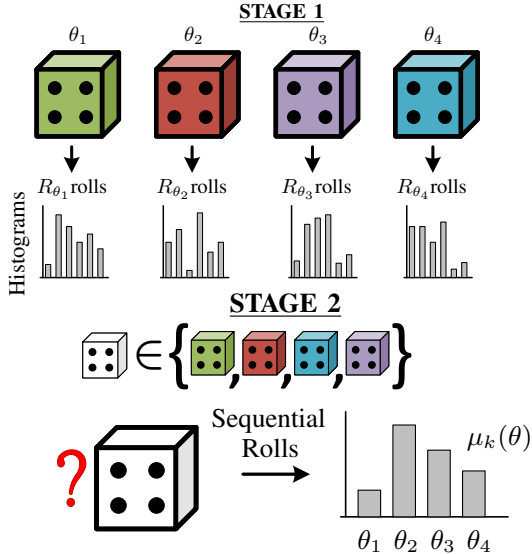
	
	The contributions of this paper are as follows: We propose a new belief update rule that allows a network of agents to collectively agree on the set of hypotheses with prior evidence that is consistent with the state of the world. Moreover, this new update rule can be implemented over networks with time{-}varying and directed links, without requiring doubly stochastic weights, or balanced graphs. To do so, we proposed an uncertain likelihood ratio function that takes into account the uncertainty generated by the finite evidence when building the statistical models for each of the hypothesis. We show that the proposed update rule converges to the distributed likelihood ratio test of the true distribution of the state of the world given the observations in the learning stage.
	
	\noindent\textbf{Notation:} Subscripts denote time indices and make use of the letter $k$. Agent indices are represented as superscripts and use the letters $i$ or $j$. The $i${-}th row and $j${-}th column entry of a matrix $A$ is denoted as $[A]_{ij}$. Moreover, for a sequence of matrices $\{A_{k}\}$ we denote $A_{k:t} = A_kA_{k{-}1}\cdots A_{t{+}1}A_t$ for $k \geq t$. We use node and agent interchangeably. 
	
	\section{Problem Formulation, Algorithm, and Results} \label{sec:pf}
	
	In this section, we describe the problem formulation of non{-}Bayesian social learning with uncertain models. Then, we present the proposed algorithm, and finally we state our main result about the convergence of the proposed method.
	
	Consider a group of $m$ agents connected over a network, with possibly time{-}varying and directed links, that seek to cooperatively identify a state of the world $\theta^* \in \Theta$ from a finite set of states (i.e., hypotheses) $\Theta = (\theta_1,\theta_2,\dots,\theta_M)$. Each agent $i$ sequentially receives independent and identically distributed symbols $\{S_k^i\}$ from a finite set $\mathcal{S} = (s_1,s_2,\dots,s_N)$, where at any time instant $k \geq 0$, the probability of agent $i$ observing symbol $s_l \in\mathcal{S}$ is $[p^i_{\theta^*}]_l$, thus $\sum_{l=1}^N [p^i_{\theta^*}]_l =1$ for all~$i$. Note that the vectors $p^i_{\theta^*}$ are unknown to agents. Moreover, the assumption of identical distribution is made for the observations of each agent, but in general different agents could observe symbols with different distributions. An agent records this information as a variable $[n^i_k]_l$ which counts the number of times the symbol $s_l$ has been observed up to time $k$.
	
	In order to identify the state of the world, each agent requires a statistical model for each hypothesis $\theta \in \Theta$. Traditionally, it is assumed each agent holds a family of distributions $\mathsf{P}^i = (p^i_{\theta} ; \theta \in \Theta )$, where a hypothesis $\theta$ being the state of the world implies that $p^i_{\theta} = p^i_{\theta^*}$. Thus, the main underlying assumption is that an agent has perfect (or \textit{dogmatic}) knowledge of the statistical model of every hypothesis. However, in this paper we \textit{relax} this assumption and incorporate uncertainty about the family of statistical models $\mathsf{P}^i$. We opt to explicitly consider that the statistical models representing the hypotheses are obtained via a finite set of prior experiments. 
	
	Every agent $i$, before starting the social learning process, has access to a realization of a set of multinomial random variables $(Z^i_\theta ; \theta \in \Theta )$, where $Z^i_\theta \sim \text{Mutinomial}(p^i_\theta,R_\theta^i)$ where $p^i_\theta$ is the vector of probabilities characterizing hypothesis $\theta$ and $R_\theta^i\geq 0$ is the number of independent trials, each of which leads to a success for one of $M$ categories. Effectively, a realization $z^i_\theta$ of the random variable $Z^i_\theta$ contains, at each coordinate, the number of times a particular symbol appeared out of $R_\theta^i$ trials. Therefore, each agent uses the tuples $\{(Z^i_\theta),R_\theta\}$ to build its uncertain statistical model. This process will be described later in Section~\ref{sec:models}.
	
	Agents interact over a possibly time{-}varying directed network, represented as a sequence of graphs $\{\mathcal{G}_k = (V,E_k)\}$, where $V = (1,\dots,m)$ is the set of nodes, and $E_k$ is the set of directed edges available at time $k$ with $(j,i)\in E_k$ implying that node $j$ can send information (to be defined later) to agent $i$ at time $k$. We denote as $d^i_k$ the out{-}degree of a node $i$ at time $k$, that is the number of nodes that node $i$ can send information to at time instant $k$. Moreover, we denote $N^i_k \subseteq V$ as the subset of nodes that node $i$ receives information from at time $k$.
	
	We propose an algorithm for the distributed cooperative learning of the state of the world $\theta^*$. As we will see later in Section~\ref{sec:models}, given the uncertain statistical models about the hypothesis set, we relax the learning condition and switch our attention to an algorithm that provides information about how much the received information is explained by each of the hypotheses.
	
	Each agent $i$ updates its belief $\mu^i_k(\theta)$ on each hypotheses $\theta \in \Theta$ at time $k$, using a new symbol $s^i_{k{+}1}$ , and the beliefs received from its incoming neighbors, i.e., $(\mu^j_k \ \text{s.t.} \ (j,i) \in E_k)$.  Therefore, each agent $i$ at any time instant $k$ has access to the tuple $\{s^i_k,n^i_k,z^i_\theta, R^i_\theta , \mu^j_k ; \theta \in\Theta, (j,i) \in E_k \}$. 
	
	We propose the following update rule:
	\begin{subequations}\label{algo:main}
		\begin{align}
		y^i_{k{+}1} & = \sum_{j \in N^i_k} \frac{y_k^j}{d_k^j{+}1}, \\
		\mu^i_{k{+}1}(\theta) & = \left( \prod_{j\in N^i_k} \mu^j_k(\theta)^{\frac{y^j_k}{d^j_k {+}1}}\ell^i_{\theta}(s^i_{k{+}1},n^i_{k{+}1})\right) ^{\frac{1}{y^i_{k{+}1}}},
		\end{align}
	\end{subequations}
	where $y^i_0 =1$ for all $i$, and 
	\begin{align}\label{eq:ell}
	\ell^i_{\theta}(s^i_{k},n^i_{k}) & =  \frac{\left( [Z_\theta]_{s_{k}^i} {+} [n_{k}^i]_{{s_{k}^i} }\right)\left(M{+}k {-}1\right)  }{\left( R_\theta^i {+} k {+}M {-}1\right) [n_{k}^i]_{s_{k}^i}    }.
	\end{align}

	\begin{remark}
		Algorithm~\ref{algo:main} is inspired by~\cite{ned15b}. However, note that beliefs $\mu^i_k(\theta)$ as beliefs, are consistent with the literature. However, they are not a probability distribution over the hypothesis space, and as such do not add up to one. Later in Section~\ref{sec:models} we will provide an explanation for this particular surrogate likelihood function~\eqref{eq:ell}.
	\end{remark}

	
	
	\begin{assumption}\label{assum:connected}
		The sequence of graphs $\left\{\mathcal{G}_k\right\}$ is $B${-}strongly{-}connected, i.e., there is an integer $B\ge 1$ such that the graph $\left(V,\bigcup_{i=kB}^{\left(k{+}1\right)B{-}1}E_i\right)$ is connected for all $k \geq 0$. 
	\end{assumption}

	Next, we state our main result about the convergence properties of the update rule~\ref{algo:main}. 
	
	\begin{theorem}\label{thm:main}
		Let Assumption~\ref{assum:connected} hold. Then, the update rule in~\ref{algo:main} has the following property:
		\begin{align*}
		\lim_{k\to \infty} \log \mu^i_k(\theta) &= \frac{1}{m} \sum_{j=1}^{m}\log \frac{\text{Dirichlet}(p_{\theta^*}^j;Z_\theta^j+\boldsymbol{1})}{\text{Dirichlet}(p_{\theta^*}^j;\boldsymbol{1})} \quad \text{a.s.} 
		\end{align*}
		where $
		\text{Dirichlet}(x ; \boldsymbol{\alpha})  =  \frac{1}{\text{Beta}(\alpha)}\prod \pi_i^{\alpha_i{-}1}$,
		and $\text{Beta}(\alpha) = {\Gamma(\sum_{i=1}^{M}\alpha_i)}/{\prod_{i=1}^{M}\Gamma(\alpha_i)}$ is the multidimensional Beta function.
	\end{theorem}

Theorem~\ref{thm:main} shows that log{-}beliefs, generated by the update rule~\ref{algo:main} on a hypothesis $\theta$ for every agent, converge asymptotically to the average log{-}likelihood ratio value of the distribution corresponding to the state of the world, with the prior empirical evidence counts as parameters. Note that this average has equal weights on the contributions of each of the agents, even{-}though we have not assumed the network is balanced. Moreover, this result also indicates that even if some agents do not have informative signals to build the statistical models for the hypothesis set, the information is aggregated over the network. Finally, note that if an agent does not have access to prior evidence $Z^i_\theta$ its isolated belief will be around $1$ for all hypothesis. As the size of the prior evidence increases, the belief on the wrong hypothesis goes to zero, and the belief on the correct hypothesis converges to a positive value. That positive value is larger for larger amounts of prior evidence.

	\section{Centralized Estimation with Uncertain Models}\label{sec:models}
	
	In this section, we analyze the effects of having uncertain models, in the sense that the family of distributions $(p^i_\theta;\theta\in \Theta)$ is not known precisely, but rather it is estimated from the realization of the random variable $Z^i_\theta$. For simplicity of exposition, we will focus on the case of a single agent, thus super indices will not be used. Later in Section~\ref{sec:convergence}, we will analyze the dynamics of the multi{-}agent case which will provide a proof for our result in Theorem~\ref{thm:main}.
	
	Consider the vector $p_\theta$ as a set of parameters to be estimated from the realization of the multinomial random variable $Z_\theta$. Recall that the probability mass function of the Multinomial distribution with parameters $p$ and $n$ evaluated at a point $x$, such that $\sum x_i =n$ and $x_i$ is a non-negative integer, is defined as $\text{Multinomial}(x; p,n)  = \frac{n!}{\prod x_i!} \prod p_i^{x_i}$.
	
	\noindent Assuming a uniform Dirichlet prior on the parameters $p_\theta$, i.e., $P(p_\theta) = \text{Dirichlet}(p_\theta ; \boldsymbol{1})$, it follows that the posterior is
	\begin{align}\label{eq:posterior}
	P(p_\theta \mid Z_\theta) & =  \text{Dirichlet}(p_\theta ; Z_\theta {+} \boldsymbol{1}).
	\end{align}
	
	Now, given the posterior distribution over the set of parameters for each of the hypotheses, one can define the surrogate likelihood function for hypotheses $\theta$ as the expected likelihood when the expectation is taken with respect to the parameter being distributed according to~\eqref{eq:posterior}, i.e.,
	\begin{align}\label{eq:hat_ell}
	\hat{\ell}_\theta(S_{k{+}1}|Z_\theta)  &= \mathbb{E}_{\nu \sim P(p_\theta \mid Z_\theta)} P_\nu(S_{k{+}1}) \nonumber \\
 	& = \int	P_\nu(S_{k{+}1}) P(\nu \mid Z_\theta) d\nu \nonumber\\
	& = \int	\nu_{S_{k{+}1}}  \text{Dirichlet}(\nu ; Z_\theta {+} \boldsymbol{1}) d\mathcal{S}_M \nonumber \\
	&= \frac{[Z_\theta]_{S_{k{+}1}}+\boldsymbol{1}}{R_\theta +M},
	\end{align}
	where $d\mathcal{S}_M $ denotes integrating $\nu = [\nu_1,\dots,\nu_M]$ with respect to the $(M{-}1)$ simplex. Therefore, under this construction, the surrogate likelihood is effectively the empirical distribution (or histogram) generated by the information provided by the random variable $Z_\theta$. More importantly, it follows from the strong law of large numbers that $
	\lim_{R_\theta \to  \infty} \hat{\ell}_\theta(S_k|Z_\theta)  = [p_\theta]_{S_{k}}, \ \forall \theta \in \Theta, \ k \geq 1$. a.s.
	
	Figure~\ref{fig:simplex} shows a geometric interpretation of the process of defining the surrogate likelihood functions. The triangle represents the probability simplex over the distributions of the signals, thus, each point in the simplex is a distribution. Initially, if no data about the hypothesis has been acquired, both hypotheses will be in the same point in the simplex, namely the point corresponding to the uniform distribution, drawn as a black dot in the middle of the simplex. As $R_{\theta_1}$ and $R_{\theta_2}$ increases, the location of the surrogate likelihoods follows a path that ends at the correct location of the distribution for the hypothesis. 
	
	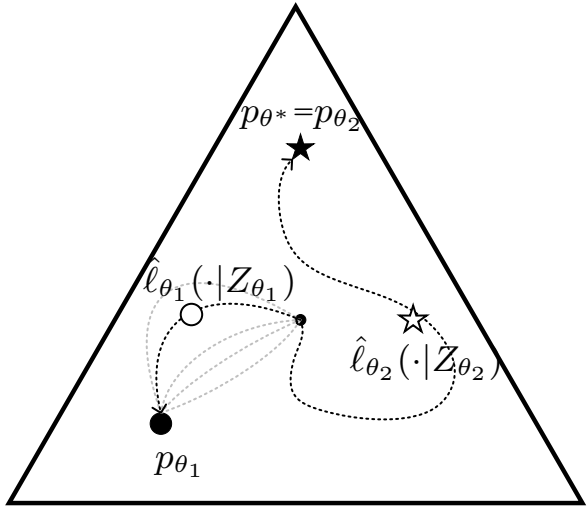
\begin{figure}[ht]
		\centering
		\resizebox{!}{0.4\textwidth}{
		\begin{overpic}[width=0.3\textwidth]{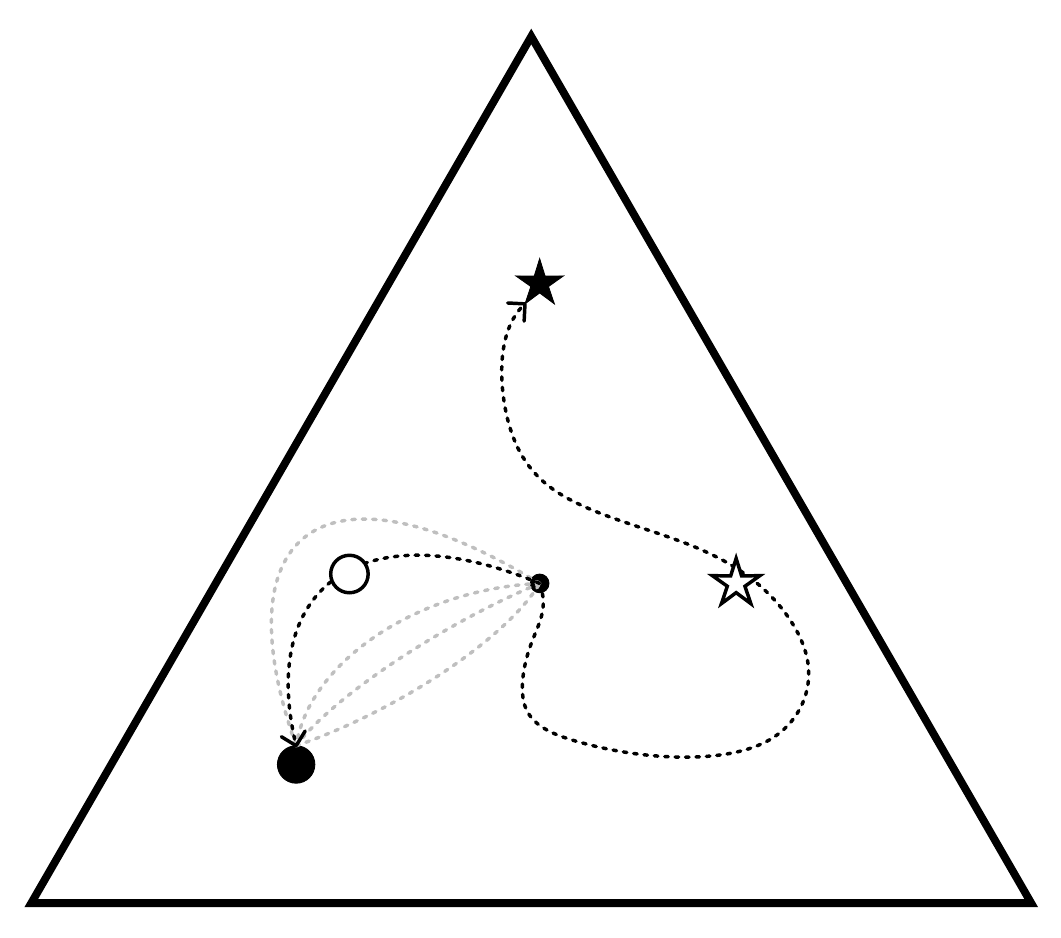}
			\put(41,66){{\small $p_{\theta^*}\texttt{=}p_{\theta_2}$}}
			\put(59,25){{\small $\hat\ell_{\theta_2}(\cdot|Z_{\theta_2})$}}
			\put(27,09){{\small $p_{\theta_1}$}}
			\put(25,38){{\small $\hat\ell_{\theta_1}(\cdot|Z_{\theta_1})$}}
		\end{overpic} 
		}
		\caption{\textbf{Geometric interpretation of the learning problem with uncertain likelihoods.} The outer triangle represents the probability simplex over the distributions of the signals $S_k$. The filled circle represents the location of the hypothesis $\theta_1$, and the filled star represents the location of the probability $\theta_2$ that is also $\theta^*$. The white filled circle represents the location of the surrogate likelihood of hypothesis $\theta_1$ for a specific realization of $Z_{\theta_1}$. The white filled star represents the location of the surrogate likelihood of hypothesis $\theta_2$ for a specific realization of $Z_{\theta_2}$. }
		\label{fig:simplex}
	\end{figure}
	
	Given the assumption that $R_{\theta_1}$ and $R_{\theta_2}$ are finite, their relative location with respect to the hypothesis $\theta^*$ might not be preserved. For example, in Figure~\ref{fig:simplex}, hypothesis $\theta_2$ is the true hypothesis. However, given some particular realizations of $Z_{\theta_1}$ and $Z_{\theta_2}$, the distribution $\hat\ell_{\theta_1}(\cdot|Z_{\theta_1})$ could be ``closer" to $p_{\theta^*}$ than the distribution $\hat\ell_{\theta_2}(\cdot|Z_{\theta_2})$. This result is formalized in the following proposition.
	
	\begin{proposition}\label{prop:not_converge}
		Consider two finite probability distributions $p_{\theta_1}$ and $p_{\theta_2}$, and their surrogate likelihood functions $\hat\ell_{\theta_1}(\cdot|Z_{\theta_1})$ and $\hat\ell_{\theta_2}(\cdot|Z_{\theta_2})$ defined as in~\eqref{eq:hat_ell}. Then, there exists non negative integers $R_{\theta_1}$ and $R_{\theta_2}$ such that 
		\begin{align*}
		Prob\big( D_{KL}(p_{\theta^*}\|\hat\ell_{\theta_1}(\cdot|Z_{\theta_1})) < D_{KL}(p_{\theta^*}\|\hat\ell_{\theta^*}(\cdot|Z_{\theta_2}))\big)   > 0.
		\end{align*} 
	\end{proposition}
	
	Proposition~\ref{prop:not_converge} can be achieved in this following example. Suppose that the agent has collected some evidence for $\theta_1$ and the path in Fig. 2 leads the white circle to a position in the upper half of the triangle. Then, if the agent has zero prior evidence for $\theta_2$, the proposition is true. 
	
	One consequence of Proposition~\ref{prop:not_converge} is that an agent using surrogate likelihoods instead of the precise likelihood functions may discard a true hypothesis with a non{-}zero probability.
	
	\begin{corollary}\label{cor:no_convergence}
		Consider the following update rule, with $\mu_0(\theta) >0$ for all $\theta \in \Theta$:
		\begin{align*}
		\mu_{k{+}1}(\theta) & = \frac{\hat\ell_{\theta}(S_{k{+}1}|Z_{\theta}) \mu_k(\theta)}{\sum_{\nu \in \Theta} \hat\ell_{\nu}(S_{k{+}1}|Z_{\nu}) \mu_k(\nu)},
		\end{align*}
		where $S_k \sim p_{\theta^*}$.  Then, there exists finite $R_{\theta_1}$ and $R_{\theta_2}$ such that $
		Prob\left(  \lim_{k \to \infty}\mu_{k}(\theta^*)  = 0\right)  > 0$.
	\end{corollary}
	
	In the next subsection, we propose a different form of surrogate likelihood function in order to overcome the probability of erroneously discarding the correct hypothesis due to the finiteness of the data from which the surrogate likelihoods are constructed.
	
	\subsection{The Uncertain Likelihood Ratio} 
	
	In order to avoid the asymptotic convergence to an erroneous hypothesis, as presented in Corollary~\ref{cor:no_convergence}, we propose the following alternative to construct a surrogate likelihood function, which we will denote as the \textit{uncertain likelihood ratio}. Initially, assume that the agent not only has access to the most recent realization of $S_{k{+}1}$, but recalls the histogram of all the signals observed so far, i.e., it has access to the counts of how many times a symbol has appeared up to time $k{+}1$. We will denote this count as $n_{k{+}1}$, where $[n_{k{+}1}]_l$ is the number of times the symbol $s_l$ has been observed.
	
	Similarly as in~\eqref{eq:hat_ell}, we will proceed to define the uncertain likelihood ratio as the posterior predictive distribution of the likelihood of the signal when the parameter is distributed according to the posterior distribution in~\eqref{eq:posterior}. However, now that we have access to $n_{k{+}1}$, we can consider $n_{k{+}1}$ as our new signal, which will be the realization of a Multinomial random variable $N_{k}$ with parameters $p_{\theta^*}$ and $k$. 
	
	Therefore, we define
	\begin{align}\label{eq:hat_ell_multi}
	&\hat{\ell}_\theta(N_{k{+}1}|Z_\theta) \nonumber = \mathbb{E}_{\nu \sim P(p_\theta \mid Z_\theta)} P_\nu(N_{k{+}1}) \nonumber \\
 	& = \int	P_\nu(N_{k{+}1}) P(\nu \mid Z_\theta) d\nu \nonumber\\
	& = \int	\text{Multinomial}(N_{k{+}1}; \nu,k)  \text{Dirichlet}(\nu ; Z_\theta {+} \boldsymbol{1}) d\mathcal{S}_M \nonumber \\
 	& = \frac{k!\Gamma(R_\theta {+} M)}{\Gamma(k {+} R_\theta {+} M)} \prod_{i=1}^{M} \frac{\Gamma([N_{k{+}1}]_i {+} [Z_{\theta}]_{i} {+}1)}{[N_{k{+}1}]_i! \Gamma([Z_{\theta}]_{i} {+}1)}.
	\end{align}

	Note that the surrogate likelihood in~\eqref{eq:hat_ell_multi} is precisely a Dirichlet{-}Multinomial distribution with parameters $Z_\theta {+}\boldsymbol{1}$. Recall that the Dirichlet{-}Multinomial probability mass function, with parameters $\alpha$ and $n$, is defined as as
	\begin{align*}
	\text{Dirichlet{-}Multinomial}(x;\alpha,n) & = \frac{n! \Gamma(\sum \alpha_i)}{\Gamma(n 
		{+} \sum \alpha_i)} \prod \frac{\Gamma(x_i{+}\alpha_i)}{x_i! \Gamma(\alpha_i)}.
	\end{align*} 
	
	Finally, we propose the uncertain likelihood ratio function as the ratio between the surrogate likelihood function in~\eqref{eq:hat_ell_multi}, and the surrogate likelihood function in~\eqref{eq:hat_ell_multi} given no prior evidence for the hypothesis set, i.e., $Z_\theta = \boldsymbol{0}$ or complete uncertainty. Thus, the uncertain likelihood function for the histogram of the signals observed up to time $k{+}1$ is defined as
	\begin{align}\label{eq:RUL}
	&\Lambda_\theta(N_{k{+}1}|Z_\theta )  = \frac{\hat{\ell}_\theta(N_{k{+}1}|Z_\theta)}{\hat{\ell}_\theta(N_{k{+}1}| \boldsymbol{0})} \nonumber\\
 	& \qquad = \frac{\text{Beta}(Z_{\theta} {+} N_{k{+}1}{+}\mathbf{1})\text{Beta}(\mathbf{1})}{\text{Beta}(Z_{\theta}{+}\mathbf{1})\text{Beta}(N_{k{+}1}{+}\mathbf{1})} \nonumber\\
	& \qquad= \frac{\text{Dirichlet{-}Multinomial}(N_{k{+}1};Z_\theta{+}\mathbf{1},k{+}1)}{\text{Dirichlet{-}Multinomial}(N_{k{+}1};\boldsymbol{1},k{+}1)}.
	\end{align}
	
	\subsection{The Asymptotic Behavior of the Ratio of Dirichlet{-}Multinomial Distributions}
	
	In this subsection, we derive some asymptotic properties of the ratio between two Dirichlet{-}Multinomial likelihood functions.
	
	Initially, consider a sequence of random variables $\{X_k \sim \text{Multinomial}(p,k) \}$, and define the random variable
	\begin{align}\label{eq:dm_ratio}
	Y_k & = \frac{\text{Dirichlet{-}Multinomial}(X_k;\alpha,n)}{\text{Dirichlet{-}Multinomial}(X_k;\beta,n)},
	\end{align}
	where $\alpha = [\alpha_1,\cdots,\alpha_M]$, and $\beta = [\beta_1,\cdots,\beta_M]$ are vectors with strictly positive entries.
	
	The next lemma describes the asymptotic behavior of the random variable $Y_k$.
	\begin{lemma}\label{lemma:ratio}
		The random variable $Y_k$ in~\eqref{eq:dm_ratio} has the following property:
		\begin{align*}
		    \lim_{k \to \infty} Y_k  =  {\text{Dirichlet}(p;\alpha{+}\boldsymbol{1})}/{\text{Dirichlet}(p;\beta{+}\boldsymbol{1})}, \ \ \text{a.s}.
		\end{align*}
	\end{lemma}
	
	\begin{proof}
		We can explicitly write the likelihood ratio as 
		\begin{align*}
		\frac{\text{Dirichlet{-}Multinomial}(X_k;\alpha,n)}{\text{Dirichlet{-}Multinomial}(X_k;\beta,n)} & = \frac{\frac{\Gamma(\sum \alpha_i)}{\Gamma(\sum \beta_i)}}{\frac{\Gamma(k {+} \sum \alpha_i)}{\Gamma(k {+} \sum \beta_i)}} \prod_{i=1}^M 
		\frac{\frac{\Gamma([X_k]_i{+}\alpha_i)}{\Gamma([X_k]_i{+}\beta_i)}}{\frac{ \Gamma(\alpha_i)}{ \Gamma(\beta_i)}}.
		\end{align*}
		
		Moreover, we can approximate the ratio of Gamma functions using Stirling's series~\cite{LAFORGIA2012833}, where 	
		{\small
			\begin{align*}
			& \Gamma\left( k {{+}} \sum_{i=1}^{M}\alpha_i\right) \bigg / \Gamma\left( k{{+}}\sum_{i=1}^{M}\beta_i\right)   \\
			&= k^{\sum\limits_{i=1}^{M}(\alpha_i\text{{-}}\beta_i)} \left(1\text{{+}}{(\sum\limits_{i=1}^{M}(\alpha_i{{-}}\beta_i))(\sum\limits_{i=1}^{M}(\alpha_i{{-}}\beta_i){{-}}1)}/{2k}\text{{+}} O({1}/{k^2})\right),
			\end{align*}
		}
		and
		\begin{align*}
		&\frac{\Gamma\left( [X_k]_i {+} \alpha_i\right) }{\Gamma\left( [X_k]_i {+} \beta_i\right) }  \\ 
		&= [X_k]_i^{\alpha_i {-}\beta_i} \left(1 {+}\frac{(\alpha_i {-}\beta_i)(\alpha_i {-}\beta_i{-}1)}{2[X_k]_i} {+} O([X_k]_i^{{-}2})\right).
		\end{align*}
		Therefore,
		\begin{align*}
		\frac{\prod_{i=1}^{M}\Gamma([X_k]_i {+}\alpha_i)}{\Gamma\left( k {+} \sum_{i=1}^{M}\alpha_i\right)} & = \prod_{i=1}^{M}  \left(\frac{[X_k]_i}{k}\right)^{\alpha_i} \cdot O\left(1 {+} \frac{1}{k}\right),
		\end{align*}
		and
		\begin{align*}
		\frac{\prod_{i=1}^{M}\Gamma([X_k]_i {+}\beta_i)}{\Gamma\left( k {+} \sum_{i=1}^{M}\beta_i\right)} & = \prod_{i=1}^{M}  \left(\frac{[X_k]_i}{k}\right)^{\beta_i} \cdot O\left(1 {+} \frac{1}{k}\right).
		\end{align*}
		
		Furthermore, it holds that
		\begin{align*}
		&\frac{\text{Dirichlet{-}Multinomial}(X_k;\alpha,k)}{\text{Dirichlet{-}Multinomial}(X_k;\beta,k)} \\
		& = \frac{\frac{\Gamma(\sum_{i=1}^M \alpha_i)}{\prod_{i=1}^{M} \Gamma(\alpha_i)} \prod_{i=1}^{M}  \left(\frac{[X_k]_i}{k}\right)^{\alpha_i} \cdot O\left(1 {+} \frac{1}{k}\right) }{ \frac{\Gamma(\sum_{i=1}^M \beta_i)}{\prod_{i=1}^{M} \Gamma(\beta_i)} \prod_{i=1}^{M}  \left(\frac{[X_k]_i}{k}\right)^{\beta_i} \cdot O\left(1 {+} \frac{1}{k}\right) } \\
		& = \frac{\frac{\Gamma(\sum_{i=1}^M \alpha_i)}{\prod_{i=1}^{M} \Gamma(\alpha_i)} \prod_{i=1}^{M}  \left(\frac{[X_k]_i}{k}\right)^{\alpha_i{-}1} \cdot O\left(1 {+} \frac{1}{k}\right) }{ \frac{\Gamma(\sum_{i=1}^M \beta_i)}{\prod_{i=1}^{M} \Gamma(\beta_i)} \prod_{i=1}^{M}  \left(\frac{[X_k]_i}{k}\right)^{\beta_i{-}1} \cdot O\left(1 {+} \frac{1}{k}\right) } \\
		& = \frac{\text{Dirichlet}(X_k/k;\alpha)\cdot O\left(1 {+} \frac{1}{k}\right)}{\text{Dirichlet}(X_k/k;\beta)\cdot O\left(1 {+} \frac{1}{k}\right)}.
		\end{align*}
		
		Finally, note that by the strong law of large number it follows that
		\begin{align*}
		\lim_{k\to \infty}  Y_k & =  \lim_{k\to \infty}\frac{\text{Dirichlet}(X_k/k;\alpha)\cdot O\left(1 {+} \frac{1}{k}\right)}{\text{Dirichlet}(X_k/k;\beta)\cdot O\left(1 {+} \frac{1}{k}\right)}\\
		& = \frac{\text{Dirichlet}(p;\alpha)}{\text{Dirichlet}(p;\beta)}, \quad \text{a.s.}
		\end{align*}
	\end{proof}
	
	It immediately follows from Lemma~\ref{lemma:ratio} and~\eqref{eq:RUL} that
	\begin{align}\label{eq:main_aux}
	\lim_{k \to \infty} \Lambda_{\theta} (N_{k}|Z_\theta ) & = \frac{\text{Beta}(\mathbf{1})}{\text{Beta}(Z_{\theta}{+}\boldsymbol{1})} \prod_{i=1}^M [p_{\theta^*}]_i^{[Z_{\theta}]_i}, \qquad \text{a.s.}
	\end{align}
	
	\subsection{An Iterative Representation of the Uncertain Likelihood Ratio }
	
	We have defined the uncertain likelihood ratio and analyzed its asymptotic behavior. However, the random variable $N_k$, which counts the realizations observed up to time $k$, is not independent across the time. Clearly, the counts at time $k{+}1$ depends on the counts at time $k$. This hinders the execution of a sequential algorithm that uses the most recent observations and the counts so far. Thus, the next lemma writes the ratio uncertain likelihood as a product of likelihoods.
	
	\begin{lemma}\label{lemm:iter}
		The uncertain likelihood ration in~\eqref{eq:RUL} can be expressed as $
		\Lambda_{\theta} (N_{k}|Z_\theta )  = \prod_{t=1}^{k} \ell_{\theta}(S_{t},N_{t}),
		$
		where $S_{k}$ is the symbol observed at time $k$, and
		\begin{align*}
		\ell_{\theta}(S_{t},N_{t}) & = \frac{\left( [Z_\theta]_{S_{t}} {+} [N_{t}]_{S_{t}}\right)\left(M{+}t {-}1\right)  }{\left( R_\theta {+} t {+}M {-}1\right) \left([N_{t}]_{S_{t}}  \right)  }.
		\end{align*}
	\end{lemma}
	\begin{proof}
		Initially, note that trivially it holds that
		\begin{align*}
		\Lambda_{\theta} (N_{k}|Z_\theta ) & = \prod_{t=1}^{k} \frac{\Lambda_{\theta} (N_{t}|Z_\theta )}{\Lambda_{\theta} (N_{t{-}1}|Z_\theta )}.
		\end{align*}
		Moreover
		\begin{align*}
		\frac{\Lambda_{\theta} (N_{t}|Z_\theta )}{\Lambda_{\theta} (N_{t{-}1}|Z_\theta )} & = \frac{\text{Beta}(Z_\theta {+} N_{t} {+} \boldsymbol{1})\text{Beta}( N_{t{-}1} {+} \boldsymbol{1})}{\text{Beta}( N_{t} {+} \boldsymbol{1}) \text{Beta}(Z_\theta {+} N_{t{-}1} {+} \boldsymbol{1})}.
		\end{align*}
		
		Also, note that
		\begin{align*}
		\frac{\text{Beta}(Z_\theta {+} N_t {+}\boldsymbol{1})}{\text{Beta}(Z_\theta {+} N_{t{-}1} {+}\boldsymbol{1})} & = \frac{\frac{\prod \Gamma([Z_\theta]_{S_{t}} {+} [N_t]_{S_{t}} {+}1)}{\Gamma(R_\theta{+}t{+}M)}}{\frac{\prod \Gamma([Z_\theta]_{S_{t}} {+} [N_{t{-}1}]_{S_{t}} {+}1)}{\Gamma(R_\theta {+} t{-}1{+}M)}} \\ & = \frac{[Z_\theta]_{S_t}{+} [N_{t}]_{S_{t}}}{(R_\theta {+} t{+}M{-}1)},
		\end{align*}
		and
		\begin{align*}
		\frac{\text{Beta}(N_t {+}\boldsymbol{1})}{\text{Beta}(N_{t{-}1} {+}\boldsymbol{1})}   & = \frac{\frac{\prod \Gamma(N_t^i {+}1)}{\Gamma(t{+}M)}}{\frac{\prod \Gamma(N_{t{-}1}^i {+}1)}{\Gamma(t{-}1{+}M)}} = \frac{[N_{t}]_{S_t}}{( t{+}M{-}1)}.
		\end{align*}
		Note that we have used the fact that $[N_t]_l = [N_{t-1}]_l$ for $l \neq S_t$,  otherwise, $[N_t]_{S_t} = [N_{t-1}]_{S_t}+1$.
	\end{proof}
	
 	Moreover, the following result follows from Lemma~\ref{lemm:iter}.
	\begin{corollary}\label{corr:limit1}
		The random variable $\ell^i_{\theta}(S_{t},N_{t})$ has the following property:	
		$
		\lim_{k \to \infty}\ell^i_{\theta}(S_{k},N_{k}) = 1 $ almost surely.
	\end{corollary}
	\begin{proof}
		The desired result follows from the strong law of large numbers and the definition of $\ell^i_{\theta}(S_{t},N_{t})$.
	\end{proof}
	
	With this results at hand, in the next section we prove our main result in Theorem~\ref{thm:main}.
	We show that a belief update rule based on the derived uncertain likelihood ratio  will converge to a value that is proportional to the probability of observing the prior empirical evidence signal $Z_\theta$ under the unknown probability law of the state of the world $\theta^*$.
	
	\section{Convergence of Non{-}Bayesian Social Learning with Uncertain Models}\label{sec:convergence}
	
	In this section, we provide the proof for our main result in  Theorem~\ref{thm:main}, which states that the proposed algorithm generates a sequence of beliefs, such that the belief fn a hypothesis converges asymptotically to the average log-likelihood ratio of the true distribution of the observations, given the empirical evidence for that specific hypothesis.
	
	Initially, we recall a number of auxiliary lemmas that will help us build the proof of Theorem~\ref{thm:main}.
	
	\begin{lemma}[Corollary 2.a in \cite{ned13}]\label{lemma_angelia}
		Let the graph sequence $\{\mathcal{G}_k\}$, with $\mathcal{G}_k = \left(E_k,V\right)$ be uniformly strongly connected, and define the matrix $A_k$ as
		\begin{align*}
		\left[A_k\right]_{ij} & = \begin{cases}
		\frac{1}{d_k^j{+}1} & \text{if } (j,i) \in E_k,\\
		0 & \text{otherwise}.
		\end{cases}
		\end{align*}
		Then, there is a sequence $\{\boldsymbol{\phi}_k\}$ of stochastic vectors such that, $
		|\left[A_{k:t}\right]_{ij} {-} \phi_k^i|  \leq C \lambda^{k{-}t}  \ \ \  \text{for all } \ k \geq t \geq 0,
		$
		for $C$ is a positive constant and $\lambda \in \left(0,1\right)$.
	\end{lemma}
	\begin{lemma}[Corollary 2.b in \cite{ned13}] \label{lemma_deltabound}
		Let the graph sequence $\left\{\mathcal{G}_k\right\}$ satisfy Assumption~\ref{assum:connected}. Define
		\begin{align}\label{eq:defdelta}
		\delta \triangleq\inf_{k\geq 0} \left(\min_{1\leq i\leq n}\left[A_{k:0} \mathbf{1}_m \right]_i\right).
		\end{align}
		Then, $\delta \geq 1/m^{mB}$, and if all $\mathcal{G}_k$ with $B=1$ are regular, then $\delta=1$.
		Furthermore, the sequence
		$\boldsymbol{\phi_k}$ from Lemma \ref{lemma_angelia} satisfies $\phi_k^j \geq \delta/m$ for all $k \geq 0, j = 1, \ldots, m$.
	\end{lemma}
	
	\begin{lemma}[Lemma $3.1$ in \cite{ram10}]\label{lemm:ram}
		Let $\{\gamma_k \}$ be a scalar sequence. If  $\lim_{k \to \infty} \gamma_k = \gamma$ and $0\leq \beta \leq 1$, then $\lim_{k\to \infty} \sum_{l=0}^{k} \beta^{k{-}l} \gamma_l = {\gamma}/({1{-}\beta})$.
	\end{lemma}
	
	Now, we are ready to state the proof of our main result in Theorem~\ref{thm:main}.
	
	\begin{proof}[Theorem~\ref{thm:main}]
		It follows from~(2b) that
		{\small
		\begin{align*}
		& y_{k{+}1}^i \log\big(\mu_{k{+}1}^i(\theta)\big)  = \sum_{j \in N^i_k} \frac{y_k^j\log \mu^j_k(\theta)}{d_k^j{+}1}  {+} \log \ell_\theta^i(S^i_{k{+}1},N^i_{k{+}1}) \\
		& \qquad \qquad = \sum_{j=1}^{m} [A_k]_{ij} y_k^j\log \mu^j_k(\theta) {+} \log \ell_\theta^i(S^i_{k{+}1},N^i_{k{+}1}).
		\end{align*}
		}

		By defining the new vector variables $[\boldsymbol{\varphi}_{k}(\theta)]_i = y_{k}^i \log\left(\mu_{k}^i(\theta)\right)$ and $[\boldsymbol{\mathcal{L}}_k(\theta)]_i = \log \ell_\theta^i(S^i_{k},N^i_{k})$, it holds that
		{\small
		\begin{align}\label{eq:vec_phi}
		\boldsymbol{\varphi}_{k{+}1}(\theta) & = A_k \boldsymbol{\varphi}_k(\theta) {+} \boldsymbol{\mathcal{L}}_{k{+}1}(\theta) \nonumber\\
		&  = A_{k:0}\boldsymbol{\varphi}_{0}(\theta)  {+} \sum_{t=1}^{k}A_{k:t}\boldsymbol{\mathcal{L}}_{t}(\theta) {+} \boldsymbol{\mathcal{L}}_{k{+}1}(\theta).
		\end{align}
		}
		
		Add and subtract $\Sigma_{t=1}^{k} \phi_k \boldsymbol{1}' \boldsymbol{\mathcal{L}}_{t}(\theta)$ from~\eqref{eq:vec_phi}, then
		{\small
		\begin{align*}
		\boldsymbol{\varphi}_{k{+}1}(\theta) &  = A_{k:0}\boldsymbol{\varphi}_{0}(\theta)  {+} \sum_{t=1}^{k}A_{k:t}\boldsymbol{\mathcal{L}}_{t}(\theta) {+} \boldsymbol{\mathcal{L}}_{k{+}1}(\theta) 
		\\ & \quad {-} \sum_{t=1}^{k} \phi_k \boldsymbol{1}' \boldsymbol{\mathcal{L}}_{t}(\theta) {+}  \sum_{t=1}^{k} \phi_k \boldsymbol{1}' \boldsymbol{\mathcal{L}}_{t}(\theta) \\
		& = \sum_{t=1}^{k}D_{k:t}\boldsymbol{\mathcal{L}}_{t}(\theta) {+} \boldsymbol{\mathcal{L}}_{k{+}1}(\theta) 
		{+}  \sum_{t=1}^{k} \phi_k \boldsymbol{1}' \boldsymbol{\mathcal{L}}_{t}(\theta),
		\end{align*}
		}
		\noindent where we have assumed that without loss of generality that $\mu_0^i(\theta)=1$ for all $i\in V$, and $D_{k:t} = A_{k:t}{-}\phi_k \boldsymbol{1}'$.
		
		Similarly, note that $y^i_{k{+}1} = [D_{k:0} \boldsymbol{1}]_i {+} \phi_k^i m$.
		
		Therefore, we have that
		{\small
			\begin{align*}
			\log \mu^i_{k{+}1}(\theta) & = \frac{\sum\limits_{t=1}^{k}[D_{k:t}\boldsymbol{\mathcal{L}}_{t}(\theta)]_i {+} [\boldsymbol{\mathcal{L}}_{k{+}1}(\theta)]_i
				{+}  \sum\limits_{\tau=1}^{t{-}1} \phi_k^i \boldsymbol{1}' \boldsymbol{\mathcal{L}}_{t}(\theta)}{[D_{k:0} \boldsymbol{1}]_i {+} \phi_k^i m},
			\end{align*}
		}
		
		\noindent and by adding and subtracting we obtain
		\allowdisplaybreaks
		
		\vspace{-0.4cm}
		{\small
			\begin{align*}
			& \log \mu^i_{k{+}1}(\theta)  = \frac{\sum\limits_{t=1}^{k}[D_{k:t}\boldsymbol{\mathcal{L}}_{t}(\theta)]_i {+} [\boldsymbol{\mathcal{L}}_{k{+}1}(\theta)]_i
				{+}  \sum\limits_{t=1}^{k} \phi_k^i \boldsymbol{1}' \boldsymbol{\mathcal{L}}_{t}(\theta)}{[D_{k:0} \boldsymbol{1}]_i {+} \phi_k^i m} \\
			& \qquad
			{-}\frac{1}{m}  \sum_{t=1}^k \boldsymbol{1}' \boldsymbol{\mathcal{L}}_{t}(\theta)  {+}\frac{1}{m}   \sum_{t=1}^k \boldsymbol{1}' \boldsymbol{\mathcal{L}}_{t}(\theta) \\
			& = \frac{m \left( \sum\limits_{t=1}^{k}[D_{k:t}\boldsymbol{\mathcal{L}}_{t}(\theta)]_i {+} [\boldsymbol{\mathcal{L}}_{k{+}1}(\theta)]_i
				{+}  \sum\limits_{t=1}^{k} \phi_k^i \boldsymbol{1}' \boldsymbol{\mathcal{L}}_{t}(\theta)\right) }{ m \left( [D_{k:0} \boldsymbol{1}]_i {+} \phi_k^i m \right) }
			\\
			&\qquad {-} \frac{  \left( [D_{k:0} \boldsymbol{1}]_i {+} \phi_k^i m \right) \left( \sum\limits_{t=1}^k \boldsymbol{1}' \boldsymbol{\mathcal{L}}_{t}(\theta) \right)  }{ m \left( [D_{k:0} \boldsymbol{1}]_i {+} \phi_k^i m \right) }
			{+}\frac{1}{m}   \sum\limits_{t=1}^k \boldsymbol{1}' \boldsymbol{\mathcal{L}}_{t}(\theta) \\
			& = \frac{ \sum\limits_{t=1}^{k}[D_{k:t}\boldsymbol{\mathcal{L}}_{t}(\theta)]_i {+} [\boldsymbol{\mathcal{L}}_{k{+}1}(\theta)]_i
			}{  [D_{k:0} \boldsymbol{1}]_i {+} \phi_k^i m  }  {-} \frac{  [D_{k:0} \boldsymbol{1}]_i   \left( \sum\limits_{t=1}^k \boldsymbol{1}' \boldsymbol{\mathcal{L}}_{t}(\theta) \right)  }{ m \left( [D_{k:0} \boldsymbol{1}]_i {+} \phi_k^i m \right) } \\
			& \qquad
			{+}\frac{1}{m}   \sum\limits_{t=1}^k \boldsymbol{1}' \boldsymbol{\mathcal{L}}_{t}(\theta).
			\end{align*}
		}
		
		Thus,
		{\small
			\begin{align*}
			& \bigg | \log \mu^i_{k{+}1}(\theta) {-} \frac{1}{m}   \sum_{t=1}^k \boldsymbol{1}' \boldsymbol{\mathcal{L}}_{t}(\theta) \bigg | \\ 
			& = \bigg |\frac{ \sum\limits_{t=1}^{k}[D_{k:t}\boldsymbol{\mathcal{L}}_{t}(\theta)]_i {+} [\boldsymbol{\mathcal{L}}_{k{+}1}(\theta)]_i
			}{  [D_{k:0} \boldsymbol{1}]_i {+} \phi_k^i m } {-}
			\frac{  [D_{k:0} \boldsymbol{1}]_i   \left( \sum\limits_{t=1}^k \boldsymbol{1}' \boldsymbol{\mathcal{L}}_{t}(\theta) \right)  }{ m \left( [D_{k:0} \boldsymbol{1}]_i {+} \phi_k^i m \right) }  \bigg |\\
			& \leq  \frac{1}{\delta}  \sum_{t=1}^{k}\bigg | [D_{k:t}\boldsymbol{\mathcal{L}}_{t}(\theta)]_i\bigg |  {+} \frac{1}{\delta}\bigg | [\boldsymbol{\mathcal{L}}_{k{+}1}(\theta)]_i \bigg | \\
			& \qquad {+} \frac{1}{m\delta}\bigg |   [D_{k:0} \boldsymbol{1}]_i   \left( \sum_{t=1}^k \boldsymbol{1}' \boldsymbol{\mathcal{L}}_{t}(\theta) \right) \bigg | \\
			& \leq  \frac{1}{\delta}  \sum_{t=1}^{k} \lambda^{k{-}t} \|\boldsymbol{\mathcal{L}}_{t}(\theta)\|_1  {+} \frac{1}{\delta}\bigg | [\boldsymbol{\mathcal{L}}_{k{+}1}(\theta)]_i \bigg | {+} \frac{1}{\delta} \lambda^{k}   \sum_{t=1}^{k}   \|\boldsymbol{\mathcal{L}}_{t}(\theta)\|_1
			\end{align*}
		}
	where we have used Lemma~\ref{lemma_angelia} to bound $D_{k:t}$ and obtain ${\left[ D_{k:0}\mathbf{1}\right]_i+ \phi_k^i m \geq \delta}$. Note that, $\delta \geq 1/m^{mB}$ from Lemma~\ref{lemma_deltabound}. 
		
		Also, note that $[\boldsymbol{\mathcal{L}}_{k{+}1}(\theta)]_i $ is upper bounded, thus, it follows from Lemma \ref{lemm:ram}, that $
		\lim_{k \to \infty} \sum_{t=1}^{k} \lambda^{k{-}t} \|\boldsymbol{\mathcal{L}}_{t}(\theta)\|_1 = 0, $ almost surely, 
		and $
		\lim_{k \to \infty}  [\boldsymbol{\mathcal{L}}_{k{+}1}(\theta)]_i  = 0, $ almost surely.
		
	Furthermore, note that $[\boldsymbol{\mathcal{L}}_{k{+}1}(\theta)]_i $ is upper bounded, thus,
		\begin{align*}
		\lim_{k\to \infty} \lambda^{k}   \sum_{t=1}^{k}   \|\boldsymbol{\mathcal{L}}_{t}(\theta)\|_1 = 0  \qquad \text{a.s.}
		\end{align*}
		
		Finally, it follows from the Lemma~\ref{lemm:iter} that
		\begin{align*}
		\lim_{k \to \infty}\frac{1}{m}   \sum_{t=1}^k \boldsymbol{1}' \boldsymbol{\mathcal{L}}_{t}(\theta) & =\lim_{k \to \infty} \log\left( \prod_{j=1}^{m}\Lambda_{\theta} (N_{k}^j|Z_\theta^j ) ^{1/m}\right)
		\end{align*}
		
		The desired result follows continuity of the logarithm function and~\eqref{eq:main_aux}.
	\end{proof}

	\section{Numerical Analysis}\label{sec:simulations}
	
	In this section, we validate the convergence properties of our proposed algorithm. Assume the agents are connected over the graph shown in Figure~\ref{fig:graph}~\cite{ned16}, which has been shown to be a pathological case of a graph that satisfies Assumption~\ref{assum:connected}. The agents' receive a private signal at each time step drawn from $K=2$ categories and their goal is to infer the $\theta \in \boldsymbol{\Theta}=\{\theta_1,\theta_2\}$ that best describes the ground truth $\theta^*=\theta_2$. 
	
	\begin{figure}[h]
\centering
	\includegraphics[width=0.8\columnwidth]{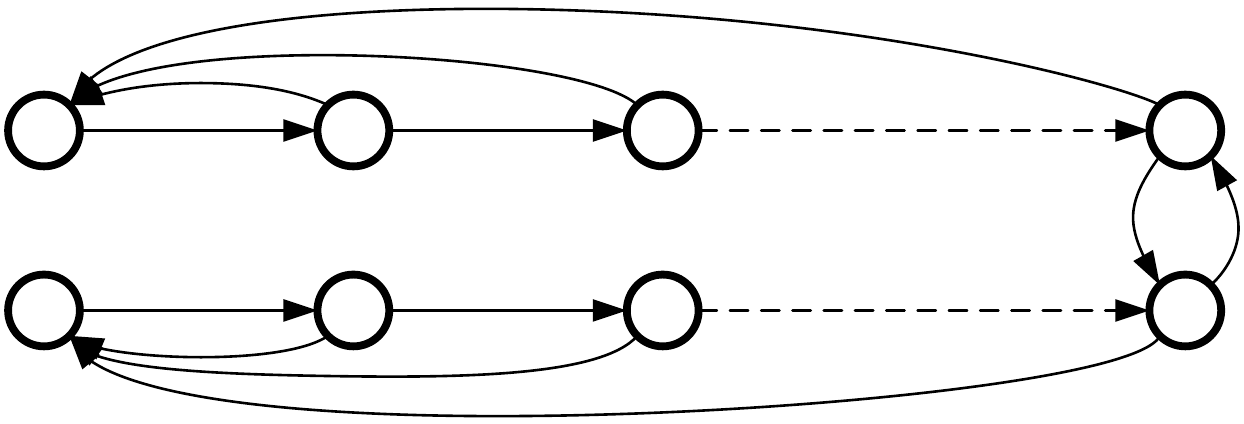}
	\caption{A directed graph with large mixing time.}
	\label{fig:graph}
\end{figure}
	
	Agents collect prior evidence $R_\theta^i$ for each hypothesis within two categories, Low, i.e. $R_\theta^i\in[0, 100]$, and High, i.e. $R_\theta^i\in[1000, 10000]$. Then, the agents observe measurements drawn from the distribution $p_{\theta^*}^i$ and update their belief as in~\eqref{algo:main}. We run $N=10$ Monte Carlo simulations and the average difference between the agents beliefs and~\eqref{eq:main_aux} is evaluated.
	
	\begin{figure}[t]
\centering
	\subfigure[$\theta_1\ne\theta^*$ convergence]{
		\includegraphics[width=0.9\columnwidth]{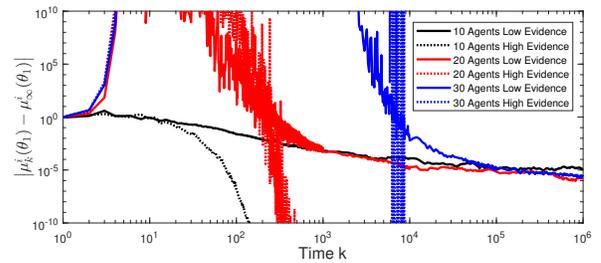}\label{fig:t1}
	}
	
    \subfigure[$\theta_2=\theta^*$ convergence]{
		\includegraphics[width=0.9\columnwidth]{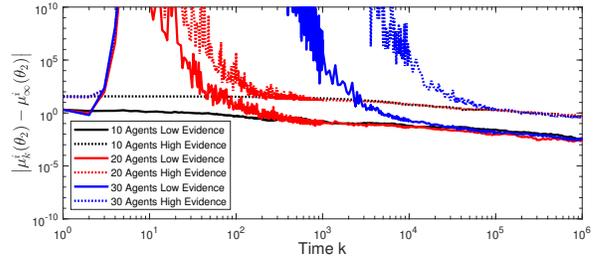}\label{fig:t2}
}
	\caption{Simulation results of Algorithm (\ref{algo:main}) for the graph in Fig. \ref{fig:graph}.}  
\end{figure}
	
	Figures \ref{fig:t1} and \ref{fig:t2} shows the result for $10$, $20$, and $30$ agents with both Low and High evidence. Figure \ref{fig:t1} shows that when the agents have a low amount of prior evidence, the same decreasing behavior is seen by all network sizes once the uncertain likelihood updates \eqref{eq:ell} converge to a value close to $1$. This also shows that the transition time for this to occur increases as the number of agents increases, which supports recent theoretical evidence that the network influence is transient in nature~\cite{olshevsky2018robust}. Furthermore, as the agents evidence increases, the rate of convergence dramatically increases and becomes exponential. Figure \ref{fig:t2} provides the same result, except this is for the ground truth hypothesis. Here, the lower the amount of evidence results in a faster rate of convergence than a high amount of prior evidence. This is because as the amount of prior evidence increases, \eqref{eq:main_aux} becomes larger and it takes longer for the beliefs to reach the convergence point. 
	
	The beliefs of a hypothesis that are consistent with the state of the world will converge to a value greater than $0$, while the beliefs of the remaining hypotheses diverge to $-\infty$. This result is seen in Table \ref{table:weak} for all network sizes with high evidence. While, when the agents have low prior evidence, the beliefs of the state of the world converge to a value close to $0$ as predicted. Thus, as the amount of evidence increases, the agents become more certain of the hypothesis that best describes the state of the world. 
	
	
\begin{table}[] 
\centering
\caption{Average point of convergence, i.e. $\frac{1}{m}\sum_{i=1}^m \log\left(\mu_T^i(\theta)\right)$}
\begin{tabular}{ccccc}
 & \multicolumn{4}{c}{Prior Evidence} \\
 \multicolumn{1}{c|}{} & \multicolumn{2}{c|}{\textbf{Low}} & \multicolumn{2}{c|}{\textbf{High}}  \\
 \multicolumn{1}{c|}{} & $\theta_1$ & \multicolumn{1}{c|}{$\theta_2$} & $\theta_1$ & \multicolumn{1}{c|}{$\theta_2$} \\ \cline{1-5} 
 \multicolumn{1}{c|}{10 Agents} & -8.28 & \multicolumn{1}{c|}{0.96} & -645 & \multicolumn{1}{c|}{3.55}  \\
 \multicolumn{1}{c|}{20 Agents} & -8.61 & \multicolumn{1}{c|}{1.03} & -659 & \multicolumn{1}{c|}{3.42} \\
\multicolumn{1}{c|}{30 Agents} & -8.05 & \multicolumn{1}{c|}{1.05} & -648 & \multicolumn{1}{c|}{3.46}  \\ 
\end{tabular}
\label{table:weak}
\vspace{-10pt}
\end{table}

	
	\section{Conclusions}\label{sec:conclusions}
	
	We proposed a new algorithm for non{-}Bayesian social learning over time{-}varying directed graphs and uncertain models. Contrary to existing literature, we analyze the effects of uncertainty in the statistical models of the hypotheses when they are built from empirical and finite evidence. We show that classical algorithms will select wrong hypotheses with non{-}zero probability. The proposed algorithm is shown to converge to the average value of a log{-}likelihood ratio between the unknown distribution of the state of the world given the empirical evidence. Moreover, doubly stochastic weights are not required, and the proposed method converges to the mean of the log{-}likelihood ratios among all the nodes in the network. Future work requires a study of convergence rates and the effects on uncertainty in the non{-}asymptotic performance of cooperative learning with uncertain models. Furthermore, it is necessary to study the explicit effects of the network topology on the convergence rates.
	
	\bibliographystyle{IEEEtran} 
	
	\bibliography{all_refs,references}
	
\end{document}